\documentclass{article}
\usepackage{cite}
\usepackage{amsmath,amssymb,amsthm}
\usepackage{titlesec,hyperref}
\usepackage{color}
\usepackage{fancyhdr}
\usepackage[margin=2.5cm]{geometry}

\newtheorem{theo}{Theorem}[section]
\newtheorem{lemm}[theo]{Lemma}
\newtheorem{defi}[theo]{Definition}

\newtheorem{prop}[theo]{Proposition}
\newtheorem{rema}[theo]{Remark}
\numberwithin{equation}{section}

\allowdisplaybreaks

\begin{document}

\title{On the Cauchy problem for the Hunter-Saxton equation on the line}
\author{
    Weikui $\mbox{Ye}^1$ \footnote{email: 904817751@qq.com} \quad and\quad
    Zhaoyang $\mbox{Yin}^{1,2}$ \footnote{email: mcsyzy@mail.sysu.edu.cn}\\
    $^1\mbox{Department}$ of Mathematics, Sun Yat-sen University,\\
    Guangzhou, 510275, China\\
    $^2\mbox{Faculty}$ of Information Technology,\\
    Macau University of Science and Technology, Macau, China
}
\date{}
\maketitle
\begin{abstract}
   In this paper, we consider the Cauchy problem for the Hunter-Saxton (HS) equation on the line. Firstly, we establish the local well-posedness for the integral form of the (HS) equation by constructing some special spaces $E^s_{p,r}$, which mix Lebesgue spaces and homogeneous Besov spaces. Then we present a global existence result and provide a sufficient condition for strong solutions to blow up in finite time for the equation. Finally, we give the ill-posedness and the unique continuation of the Hunter-Saxton equation.
\end{abstract}

\tableofcontents

\section{Introduction}
\par
In this paper, we consider the following Hunter-Saxton $(HS)$ equation:
\begin{align}\label{hseq}
(u_t+uu_x)_x=\frac12u_x^2.
\end{align}
 The equation (\ref{hseq}) was derived by Hunter and Saxton as an asymptotic model of liquid crystals \cite{hsd,Beals}. The $HS$ equation is completely integrable \cite{Beals,Zheng1} and has a bi-Hamiltonian structure \cite{Olver}. Local well-posedness and blow-up phenomena for the Cauchy problems of the (HS) equation on the circle were studied in \cite{hsd,yzhs}. Global weak solutions of the (HS) equation was investigated in \cite{b-c1,jmpa}.

  In particular, the (HS) equation is the limit form of the well-known Camassa-Holm (CH) equation:
$$(1-\partial_x^2)u_t=3uu_x-2u_x u_{xx}-u u_{xxx}.$$
Local well-posedness and ill-posedness for the Cauchy problem of the (CH) equation were investigated in \cite{c-e2,G-L-Y,Danch}.
Blow-up phenomena and global existence of strong solutions were discussed in \cite{C-E,c2,c-e2,c-e3}. The existence of global weak solutions and dissipative solutions were also investigated in \cite{X-Z,B-C-Z1,B-C-Z2}. Moreover, the existence and uniqueness of the global conservative solutions on the line were studied in \cite{C-M,uniqueness1}. While the existence and uniqueness of the global conservative solutions on the circle were studied in \cite{hr,uniqueness2}.

However, the Cauchy problem of the Hunter-Saxton on the line has not been studied yet. In this paper we consider the Cauchy problem for the Hunter-Saxton (HS) on the line:
\begin{equation}\label{eq0}
  \left\{\begin{array}{l}
    u_{xt}+(uu_x)_x=\frac{1}{2}u_x^2,  \\
    u(0,x)=u_0(x), \quad x\in\mathbb{R}.
  \end{array}\right.
\end{equation}
Taking $\int_{\infty}^{x} dz$ to the Hunter-Saxton equation (\ref{eq0}) and letting $g(t)$ is a bounded function, we get
\begin{equation}\label{eq1}
  \left\{\begin{array}{l}
    u_{t}+uu_x=\int_{-\infty}^{x}\frac{1}{2}u_x^2(z)dz+g(t),  \\
    u(0,x)=u_0(x), \quad x\in\mathbb{R}.
  \end{array}\right.
\end{equation}

In fact, the mainly difficulty is that the term $\int_{-\infty}^{x}\frac{1}{2}u_x^2(z)dz$ is not bounded in the inhomogeneous Besov spaces $B^s_{p,r}$, even the Lebesgue spaces $L^p$ $(1\leq p<\infty)$. Since $u^2_x\geq 0$ and $\int_{-\infty}^{x}\frac{1}{2}u_x^2(z)dz$ is monotonic increasing, one can't consider \eqref{eq1} in $B^s_{p,r}$ on the line. To overcome this difficulty, we would like to study \eqref{eq1} in some new spaces $E^s_{p,r}$ which mix Lebesgue spaces and homogeneous Besov spaces. In this way, one will get a bounded iterative sequence in $E^s_{p,r}$ and finally obtain a solution for the Cauchy problem of \eqref{eq1}.

The remaining part of the paper is organized as follows. In Section 2, we introduce some useful preliminaries. In Section 3, we prove the local well-posedness of (\ref{eq1}) in some special space $E^s_{p,r}$, the main approach is based on the Littlewood-Paley theory and transport equations theory. In Section 4, we obtain a global existence result and give a sufficient condition for strong solutions to blow up in finite time. In Section 5, we give the ill-posedness of (\ref{eq1}) in the special Besov space $L^{\infty}\cap \dot{B}^1_{p,r}\cap \dot{B}^{1+\frac{d}{p}}_{p,r},r>1$. In Section 6, we give the unique continuation of (\ref{eq1}) when $g(t):=C\int_{-\infty}^{x}u_x^2dz$.

\section{Preliminaries}
\par
In this section, we will recall some propositons on the Littlewood-Paley decomposition and Besov spaces.
\begin{prop}\cite{book}
Let $\mathcal{C}$ be the annulus $\{\xi\in\mathbb{R}^d:\frac 3 4\leq|\xi|\leq\frac 8 3\}$. There exist radial functions $\chi$ and $\varphi$, valued in the interval $[0,1]$, belonging respectively to $\mathcal{D}(B(0,\frac 4 3))$ and $\mathcal{D}(\mathcal{C})$, and such that
$$ \forall\xi\in\mathbb{R}^d,\ \chi(\xi)+\sum_{j\geq 0}\varphi(2^{-j}\xi)=1, $$
$$ \forall\xi\in\mathbb{R}^d\backslash\{0\},\ \sum_{j\in\mathbb{Z}}\varphi(2^{-j}\xi)=1, $$
$$ |j-j'|\geq 2\Rightarrow\mathrm{Supp}\ \varphi(2^{-j}\cdot)\cap \mathrm{Supp}\ \varphi(2^{-j'}\cdot)=\emptyset, $$
$$ j\geq 1\Rightarrow\mathrm{Supp}\ \chi(\cdot)\cap \mathrm{Supp}\ \varphi(2^{-j}\cdot)=\emptyset. $$
The set $\widetilde{\mathcal{C}}=B(0,\frac 2 3)+\mathcal{C}$ is an annulus, and we have
$$ |j-j'|\geq 5\Rightarrow 2^{j}\mathcal{C}\cap 2^{j'}\widetilde{\mathcal{C}}=\emptyset. $$
Further, we have
$$ \forall\xi\in\mathbb{R}^d,\ \frac 1 2\leq\chi^2(\xi)+\sum_{j\geq 0}\varphi^2(2^{-j}\xi)\leq 1, $$
$$ \forall\xi\in\mathbb{R}^d\backslash\{0\},\ \frac 1 2\leq\sum_{j\in\mathbb{Z}}\varphi^2(2^{-j}\xi)\leq 1. $$
\end{prop}

Denote $\mathcal{F}$ by the Fourier transform and $\mathcal{F}^{-1}$ by its inverse.
Let $u$ be a tempered distribution in $\mathcal{S}'_h(\mathbb{R}^d)$. For all $j\in\mathbb{Z}$, define
$$
\dot{\Delta}_j u=\mathcal{F}^{-1}(\varphi(2^{-j}\cdot)\mathcal{F}u)\,\ \quad \dot{S}_ju=\sum_{j'<j}\Delta_{j'}u.
$$
Then the Littlewood-Paley decomposition is given as follows:
$$ u=\sum_{j\in\mathbb{Z}}\Delta_j u \quad \text{in}\ \mathcal{S}'_h(\mathbb{R}^d). $$

Let $s\in\mathbb{R},\ 1\leq p,r\leq\infty.$ The homogeneous Besov space $\dot{B}^s_{p,r}(\mathbb{R}^d)$ is defined by
$$ \dot{B}^s_{p,r}=\dot{B}^s_{p,r}(\mathbb{R}^d)=\{u\in S'_h(\mathbb{R}^d):\|u\|_{\dot{B}^s_{p,r}(\mathbb{R}^d)}=\Big\|(2^{js}\|\dot{\Delta}_j u\|_{L^p})_j \Big\|_{l^r(\mathbb{Z})}<\infty\}. $$

Let $u$ be a tempered distribution in $\mathcal{S}'(\mathbb{R}^d)$. For all $j\in\mathbb{Z}$, define
$$
\Delta_j u=0\,\ \text{if}\,\ j\leq -2,\quad
\Delta_{-1} u=\mathcal{F}^{-1}(\chi\mathcal{F}u),\quad
\Delta_j u=\mathcal{F}^{-1}(\varphi(2^{-j}\cdot)\mathcal{F}u)\,\ \text{if}\,\ j\geq 0,\quad
S_j u=\sum_{j'<j}\Delta_{j'}u.
$$
Then the Littlewood-Paley decomposition is given as follows:
$$ u=\sum_{j\in\mathbb{Z}}\Delta_j u \quad \text{in}\ \mathcal{S}'(\mathbb{R}^d). $$

Let $s\in\mathbb{R},\ 1\leq p,r\leq\infty.$ The nonhomogeneous Besov space $B^s_{p,r}(\mathbb{R}^d)$ is defined by
$$ B^s_{p,r}=B^s_{p,r}(\mathbb{R}^d)=\{u\in S'(\mathbb{R}^d):\|u\|_{B^s_{p,r}(\mathbb{R}^d)}=\Big\|(2^{js}\|\Delta_j u\|_{L^p})_j \Big\|_{l^r(\mathbb{Z})}<\infty\}. $$

There are some properties about Besov spaces. For the homogeneous Besov space, we have
\begin{prop}\label{prop0}\cite{book,he}
Let $s\in\mathbb{R},\ 1\leq p,p_1,p_2,r,r_1,r_2\leq\infty$ with $s<\frac{d}{p} or s=\frac{d}{p}, r=1$.  \\
(1) $\dot{B}^s_{p,r}$ is a Banach space, and is continuously embedded in $\mathcal{S}'_h$. \\
(2) If $r<\infty$, then $\lim\limits_{j\rightarrow\infty}\|S_j u-u\|_{\dot{B}^s_{p,r}}=0$.\\
(3) If $p_1\leq p_2$ and $r_1\leq r_2$, then $\dot{B}^s_{p_1,r_1}\hookrightarrow \dot{B}^{s-d(\frac 1 {p_1}-\frac 1 {p_2})}_{p_2,r_2}. $\\
(4) If $p\in [1,2]$, then $\dot{B}^s_{p,p}\hookrightarrow L^{p}$. If $p\in [2,\infty)$, then $\dot{B}^s_{p,2}\hookrightarrow L^{p}$. If $p=\infty$, then $\dot{B}^s_{\infty,1}\hookrightarrow L^{\infty}\quad .$\\
(5) Fatou property: if $(u_n)_{n\in\mathbb{N}}$ is a bounded sequence in $\dot{B}^s_{p,r}$, then an element $u\in \dot{B}^s_{p,r}$ and a subsequence $(u_{n_k})_{k\in\mathbb{N}}$ exist such that
$$ \lim_{k\rightarrow\infty}u_{n_k}=u\ \text{in}\ \mathcal{S}'_h\quad \text{and}\quad \|u\|_{\dot{B}^s_{p,r}}\leq C\liminf_{k\rightarrow\infty}\|u_{n_k}\|_{\dot{B}^s_{p,r}}. $$
(6) Let $m\in\mathbb{R}$ and $f$ be a $S^m$-mutiplier, (i.e. f is a smooth function and satisfies that $\forall\alpha\in\mathbb{N}^d$, $\exists C=C(\alpha)$, such that $|\partial^{\alpha}f(\xi)|\leq C|\xi|^{m-|\alpha|},\ \forall\xi\in\mathbb{R}^d)$. Then the operator $f(D)=\mathcal{F}^{-1}(f\mathcal{F})$ is continuous from $B^s_{p,r}$ to $B^{s-m}_{p,r}$.\\
(7) If $s>0$, then $L^p\cap\dot{B}^s_{p,r}=B^s_{p,r}$
\end{prop}
\quad\\
For the nonhomogeneous Besov space, we have
\begin{prop}\cite{book,he}
Let $s\in\mathbb{R},\ 1\leq p,p_1,p_2,r,r_1,r_2\leq\infty.$  \\
(1) $B^s_{p,r}$ is a Banach space, and is continuously embedded in $\mathcal{S}'$. \\
(2) If $r<\infty$, then $\lim\limits_{j\rightarrow\infty}\|S_j u-u\|_{B^s_{p,r}}=0$. If $p,r<\infty$, then $C_0^{\infty}$ is dense in $B^s_{p,r}$. \\
(3) If $p_1\leq p_2$ and $r_1\leq r_2$, then $ B^s_{p_1,r_1}\hookrightarrow B^{s-d(\frac 1 {p_1}-\frac 1 {p_2})}_{p_2,r_2}. $
If $s_1<s_2$, then the embedding $B^{s_2}_{p,r_2}\hookrightarrow B^{s_1}_{p,r_1}$ is locally compact. \\
(4) If $p\in [1,2]$, then $B^s_{p,p}\hookrightarrow L^{p}$. If $p\in [2,\infty)$, then $B^s_{p,2}\hookrightarrow L^{p}$. If $p=\infty$, then $B^s_{\infty,1}\hookrightarrow L^{\infty}\quad .$\\
(5) Fatou property: if $(u_n)_{n\in\mathbb{N}}$ is a bounded sequence in $B^s_{p,r}$, then an element $u\in B^s_{p,r}$ and a subsequence $(u_{n_k})_{k\in\mathbb{N}}$ exist such that
$$ \lim_{k\rightarrow\infty}u_{n_k}=u\ \text{in}\ \mathcal{S}'\quad \text{and}\quad \|u\|_{B^s_{p,r}}\leq C\liminf_{k\rightarrow\infty}\|u_{n_k}\|_{B^s_{p,r}}. $$
(6) Let $m\in\mathbb{R}$ and $f$ be a $S^m$-mutiplier, (i.e. f is a smooth function and satisfies that $\forall\alpha\in\mathbb{N}^d$, $\exists C=C(\alpha)$, such that $|\partial^{\alpha}f(\xi)|\leq C(1+|\xi|)^{m-|\alpha|},\ \forall\xi\in\mathbb{R}^d)$.
Then the operator $f(D)=\mathcal{F}^{-1}(f\mathcal{F})$ is continuous from $B^s_{p,r}$ to $B^{s-m}_{p,r}$.
\end{prop}

We then introduce some useful lammas in nonhomogeneous Besov space(homogeneous Besov space is similar).
\begin{lemm}\label{prop}\cite{book,he}
(1) If $s_1<s_2$, $\theta \in (0,1)$, and $(p,r)$ is in $[1,\infty]^2$, then we have
$$ \|u\|_{B^{\theta s_1+(1-\theta)s_2}_{p,r}}\leq \|u\|_{B^{s_1}_{p,r}}^{\theta}\|u\|_{B^{s_2}_{p,r}}^{1-\theta}. $$
(2) If $s\in\mathbb{R},\ 1\leq p\leq\infty,\ \varepsilon>0$, a constant $C=C(\varepsilon)$ exists such that
$$ \|u\|_{B^s_{p,1}}\leq C\|u\|_{B^s_{p,\infty}}\ln\Big(e+\frac {\|u\|_{B^{s+\varepsilon}_{p,\infty}}}{\|u\|_{B^s_{p,\infty}}}\Big). $$
\end{lemm}

\begin{lemm}\label{product}\cite{book,he}
(1) For any $s>0$ and any $(p,r)$ in $[1,\infty]^2$, the space $L^{\infty} \cap B^s_{p,r}$ is an algebra, and a constant $C=C(s,d)$ exists such that
$$ \|uv\|_{B^s_{p,r}}\leq C(\|u\|_{L^{\infty}}\|v\|_{B^s_{p,r}}+\|u\|_{B^s_{p,r}}\|v\|_{L^{\infty}}). $$
(2) If $1\leq p,r\leq \infty,\ s_1\leq s_2,\ s_2>\frac{d}{p} (s_2 \geq \frac{d}{p}\ \text{if}\ r=1)$ and $s_1+s_2>\max(0, \frac{2d}{p}-d)$, there exists $C=C(s_1,s_2,p,r,d)$ such that
$$ \|uv\|_{B^{s_1}_{p,r}}\leq C\|u\|_{B^{s_1}_{p,r}}\|v\|_{B^{s_2}_{p,r}}. $$
(3) If $1\leq p\leq 2$,  there exists $C=C(p,d)$ such that
$$ \|uv\|_{B^{\frac d p-d}_{p,\infty}}\leq C \|u\|_{B^{\frac d p-d}_{p,\infty}}\|v\|_{B^{\frac d p}_{p,1}}. $$
\end{lemm}

Now we state some useful results in the transport equation theory, which are crucial to the proofs of our main theorem later.
\begin{equation}\label{transport}
\left\{\begin{array}{l}
    f_t+v\cdot\nabla f=g,\ x\in\mathbb{R}^d,\ t>0, \\
    f(0,x)=f_0(x).
\end{array}\right.
\end{equation}

\begin{lemm}\label{priori estimate}\cite{book,li2}
Let $s\in\mathbb{R},\ 1\leq p,r\leq\infty$.
There exists a constant $C$ such that for all solutions $f\in L^{\infty}([0,T];B^s_{p,r})$ of \eqref{transport} in one dimension with initial data $f_0$ in $B^s_{p,r}$, and $g$ in $L^1([0,T];B^s_{p,r})$, we have, for a.e. $t\in[0,T]$,
$$ \|f(t)\|_{B^s_{p,r}}\leq e^{CV(t)}\Big(\|f_0\|_{B^s_{p,r}}+\int_0^t e^{-CV(t')}\|g(t')\|_{B^s_{p,r}}dt'\Big) $$
with
\begin{equation*}
  V'(t)=\left\{\begin{array}{ll}
  \|\nabla v\|_{B^{s+1}_{p,r}},\ &\text{if}\ s>\max(-\frac 1 2,\frac 1 {p}-1), \\
  \|\nabla v\|_{B^{s}_{p,r}},\ &\text{if}\ s>\frac 1 {p}\ \text{or}\ (s=\frac 1 {p},\ p<\infty, \ r=1),
  \end{array}\right.
\end{equation*}
and when $s=\frac 1 p-1,\ 1\leq p\leq 2,\ r=\infty,\ \text{and}\ V'(t)=\|\nabla v\|_{B^{\frac 1 p}_{p,1}}$.
\end{lemm}

\begin{lemm}\label{existence}\cite{book}
Let $1\leq p\leq p_1\leq\infty,\ 1\leq r\leq\infty,\ s> -d\min(\frac 1 {p_1}, \frac 1 {p'})$. Let $f_0\in B^s_{p,r}$, $g\in L^1([0,T];B^s_{p,r})$, and let $v$ be a time-dependent vector field such that $v\in L^\rho([0,T];B^{-M}_{\infty,\infty})$ for some $\rho>1$ and $M>0$, and
$$
  \begin{array}{ll}
    \nabla v\in L^1([0,T];B^{\frac d {p_1}}_{p_1,\infty}), &\ \text{if}\ s<1+\frac d {p_1}, \\
    \nabla v\in L^1([0,T];B^{s-1}_{p,r}), &\ \text{if}\ s>1+\frac d {p_1}\ or\ (s=1+\frac d {p_1}\ and\ r=1).
  \end{array}
$$
Then the equation \eqref{transport} has a unique solution $f$ in \\
-the space $C([0,T];B^s_{p,r})$, if $r<\infty$, \\
-the space $\Big(\bigcap_{s'<s}C([0,T];B^{s'}_{p,\infty})\Big)\bigcap C_w([0,T];B^s_{p,\infty})$, if $r=\infty$.
\end{lemm}

\begin{lemm}\label{continuity}\cite{li2}
Let $1\leq p\leq\infty,\ 1\leq r<\infty,\ s>\frac d p\ (or \ s=\frac d p,\ p<\infty,\ r=1)$. Denote $\bar{\mathbb{N}}=\mathbb{N}\cup\{\infty\}$. Let $(v^n)_{n\in\bar{\mathbb{N}}}\in C([0,T];B^{s+1}_{p,r})$. Assume that $(f^n)_{n\in\bar{\mathbb{N}}}$ in $C([0,T];B^s_{p,r})$ is the solution to
\begin{equation}
\left\{\begin{array}{l}
    f^n_t+v^n\cdot\nabla f^n=g,\ x\in\mathbb{R}^d,\ t>0, \\
    f^n(0,x)=f_0(x)
\end{array}\right.
\end{equation}
with initial data $f_0\in B^s_{p,r},\ g\in L^1([0,T];B^s_{p,r})$ and that for some $\alpha\in L^1([0,T])$, $\sup\limits_{n\in\bar{\mathbb{N}}}\|v^n(t)\|_{B^{s+1}_{p,r}}\leq \alpha(t)$.
If $v^n \rightarrow v^{\infty}$ in $L^1([0,T];B^s_{p,r})$, then $f^n \rightarrow f^{\infty}$ in $C([0,T];B^s_{p,r})$.
\end{lemm}

\section{Local well-posedness}
\par
In this section, we establish local well-posedness of \eqref{eq1} in Besov spaces. In fact, though $u\in B^{s}_{p,r}$ $(1\leq p<\infty, s>0)$, and u is not a constant, the right hand side of (\ref{eq1})
$$\int_{-\infty}^{x}\frac{1}{2}u_x^2(z)dz+g(t)\notin B^{s}_{p,r}.$$
For example, if $u\in H^2$, let $g(t)=0,f(t,x):=\int_{-\infty}^{x}\frac{1}{2}u_x^2(z)dz$, we have
$$f_x(t,x)=\frac{1}{2}u_x^2(z)\geq 0.$$
This implies $f(t,x)\notin L^p(\mathbb{R})$ $(1\leq p<\infty)$, let along $f(t,x)\in B^{s}_{p,r}$. To overcome this difficulty, we firstly introduce the following new function spaces.
\begin{defi}
Let $s\geq 2$ and $1\leq p,r \leq\infty.$ Set
\begin{equation}
  E^s_{p,r}\triangleq\left \{\begin{array}{l}
   L^{\infty}\cap\dot{B}^{s-1}_{p,r}\cap\dot{B}^{s-2}_{p,r}\cap\dot{W}^{1,q},\quad\quad\quad when\quad s>2,  \\
   E^2_{p,r},\quad\quad\quad\quad\quad\quad\quad\quad\quad\quad\quad\quad\quad when\quad s=2,
  \end{array}\right.
\end{equation}
\begin{equation}
  \bar{E}^s_{p,r}\triangleq\left \{\begin{array}{l}
   L^{\infty}\cap\dot{B}^{s-1}_{p,r},\quad\quad\quad\quad\quad when\quad s>2,  \\
   \bar{E}^2_{p,r},\quad\quad\quad\quad\quad\quad\quad\quad when\quad s=2,
  \end{array}\right.
\end{equation}
where
\begin{equation}\label{equ1-1-5}
  E^2_{p,r}\triangleq\left \{\begin{array}{l}
   L^{\infty}\cap\dot{B}^1_{p,r}\cap\dot{B}^2_{p,r}\cap \dot{W}^{1,q},\quad r=p,\quad when\quad 1\leq p\leq2,  \\
   L^{\infty}\cap\dot{B}^1_{p,r}\cap\dot{B}^2_{p,r}\cap \dot{W}^{1,q},\quad r=2,\quad when\quad 2\leq p\leq\infty,
  \end{array}\right.
\end{equation}
\begin{equation}\label{equ1-1-5}
  \bar{E}^2_{p,r}\triangleq\left \{\begin{array}{l}
   L^{\infty}\cap\dot{B}^1_{p,r},\quad r=p,\quad when\quad 1\leq p\leq2,  \\
   L^{\infty}\cap\dot{B}^1_{p,r},\quad r=2,\quad when\quad 2\leq p\leq\infty,
  \end{array}\right.
\end{equation}
 and $\frac{1}{p}+\frac{1}{q}=1$.
\end{defi}

For example, if $s=p=r=2$, then $E^{2}_{2,2}=L^{\infty}\cap\dot{H}^1\cap\dot{H}^2$, which implies $u\in B^{\frac{3}{2}}_{\infty ,2}$ and $u_x\in H^1$ by proposition \ref{prop0} ($\dot{H}^2\hookrightarrow \dot{B}^{\frac{3}{2}}_{\infty ,2},L^{\infty}\cap\dot{B}^{\frac{3}{2}}_{\infty ,2}=B^{\frac{3}{2}}_{\infty ,2}$). If $s=2,p=r=1$, $E^{2}_{1,1}=L^{\infty}\cap\dot{B}^1_{1,1}\cap\dot{B}^2_{1,1}$, which implies $u\in B^{1}_{\infty ,1}$ and $u_x\in B^{1}_{1,1}$. Actually, when $1\leq p\leq 2$, we have $\dot{B}^{s-1}_{p,p}\cap\dot{B}^{s-2}_{p,p}\subset\dot{W}^{1,q}$, so the space $\dot{W}^{1,q}$ is unnecessary. Moreover, we can easily prove that both $E^2_{p,r}$ and $\bar{E}^s_{p,r}$ are Bananch spaces.

Here is our main result.
\begin{theo}\label{theorem}
Let $u_0\in E^s_{p,r}$. Then there exists a time $T>0$ such that \eqref{eq1} has a unique solution u in $C([0,T];E^s_{p,r})\cap C^1([0,T];\bar{E}^s_{p,r}).$ Moreover the solution depends continuously on the initial data.
\end{theo}

\begin{proof}

We use six steps to prove Theorem \ref{theorem}. Without loss of generality, we consider the case $g(t)=0$ in (\ref{eq1}) for $g(t)$ is just a bound term independent of x. And for simplicity, we first consider the case $E^2_{2,2}=L^{\infty}\cap\dot{H}^1\cap\dot{H}^2=B^{\frac{3}{2}}_{\infty ,2}\cap\dot{H}^1\cap\dot{H}^2.$\\

\textbf{Step one. Constructing approximate solution.} \\

We firstly set $u^0\triangleq0,$ and define a sequence $(u^n)_{n\in\mathbb{N}}$ of smooth functions by solving the following linear transport equations:
\begin{equation}\label{eq2}
    \left\{\begin{array}{l}
    u_t^{n+1}+u^nu_x^{n+1}=\int_{-\infty}^{x}\frac{1}{2}(u_x^n)^2(z)dz, \\
    u^{n+1}|_{t=0}=u_0.
    \end{array}\right.
\end{equation}
Define
$$G^n=\int_{-\infty}^{x}\frac{1}{2}(u_x^n)^2(z)dz.$$
We assume that $u^n\in C([0,T];E^2_{2,2})\cap C([0,T];\bar{E}^2_{2,2})$  for all $T>0$, then we have
$$\|G^n\|_{L^{\infty}}\leq C\|u^n_x\|^2_{L^2}\leq C\|u^n\|^2_{E^2_{2,2}},$$
$$\|G^n_x\|_{L^{2}}\leq C\|u^n_x\|_{L^{\infty}}\|u^n_x\|_{L^2}\leq C\|u^n\|^2_{E^2_{2,2}},$$
$$\|G^n_{xx}\|_L^{2}\leq C\|u^n_x\|_{L^{\infty}}\|u^n_{xx}\|_{L^2}\leq C\|u^n\|^2_{E^2_{2,2}}.$$
This implies that
\begin{align}\label{ineq1}
    \|G^n\|_{E^2_{2,2}}\leq C\|u^n\|^2_{E^2_{2,2}}.
\end{align}
So $G^n\in L^{\infty}([0,T];E^2_{2,2})$.
Similar to the proof of Lemma \ref{existence}, we deduce that \eqref{eq2} has a global solution $u^{n+1}$ which belongs to $C([0,T];E^2_{2,2})\cap C([0,T];\bar{E}^2_{2,2})$ for all $T>0$.
\quad\\

\textbf{Step two. Uniform bounds.}  \\

Using Lemma \ref{priori estimate}, we have
\begin{equation}\label{fuzhiineq23}
    \|u^{n+1}(t)\|_{L^{\infty}}\leq e^{C\int_0^t \|u^n\|_{E^2_{2,2}}dt'}
    \Big(\|u_0\|_{E^2_{2,2}}+\int_0^t e^{-C\int_0^{t'} \|u^n\|_{E^2_{2,2}} dt''}
    \|G^n\|_{E^2_{2,2}}dt'\Big) ,
\end{equation}
\begin{equation}\label{fuzhiineq25}
    \|u^{n+1}_x(t)\|_{H^{1}}\leq e^{C\int_0^t \|u^n\|_{E^2_{2,2}}dt'}
    \Big(\|u_0\|_{E^2_{2,2}}+\int_0^t e^{-C\int_0^{t'} \|u^n\|_{E^2_{2,2}} dt''}
    \|G^n\|_{E^2_{2,2}}dt'\Big) ,
\end{equation}
Combining (\ref{fuzhiineq23}), (\ref{fuzhiineq25}) and \eqref{ineq1}, we get
\begin{align}
    \|u^{n+1}(t)\|_{E^2_{2,2}} \leq e^{C\int_0^t \|u^n\|_{E^2_{2,2}}dt'}
    \Big(\|u_0\|_{E^2_{2,2}}+\int_0^t e^{-C\int_0^{t'} \|u^n\|_{E^2_{2,2}} dt''}
    \|u^n\|^2_{E^2_{2,2}}dt'\Big)  \notag\\.  \label{ineq3}
\end{align}
Then, we fix a $T>0$ such that $ 2C^2 T\|u_0\|_{E^2_{2,2}} <1 $ and suppose that
\begin{equation}\label{ineq4}
    \forall t\in [0,T],\ \|u^n(t)\|_{E^2_{2,2}} \leq
    \frac{C\|m_0\|_{E^2_{2,2}}}{(1-2C^2 t\|u_0\|_{E^2_{2,2}})}.
\end{equation}
Plugging \eqref{ineq4} into \eqref{ineq3} and using a simple calculation yield
\begin{align}\label{yzyj1}
    \|u^{n+1}(t)\|_{E^2_{2,2}} &\leq
    C\|u_0\|_{E^2_{2,2}}(1-2C^2 t\|u_0\|_{E^2_{2,2}})^{-\frac 12}
    \Big(1+C^2\|u_0\|_{E^2_{2,2}} \int_0^t (1-2C^2 t\|u_0\|_{E^2_{2,2}})^{-\frac 32} dt'\Big)  \\
    &\leq \frac{C\|u_0\|_{E^2_{2,2}}}{(1-2C^2 t\|u_0\|_{E^2_{2,2}})}\approx C_{u_0}.
\end{align}
Therefore, $\{u^n\}_{n\in \mathbb{N}}$ is bounded in $L^{\infty}([0,T];{E^2_{2,2}})$.
\quad\\

\textbf{Step three. Cauchy sequence.}  \\
\par
We are going to prove that $\{u^n\}_{n\in \mathbb{N}}$ is a Cauchy sequence in $L^{\infty}([0,T];L^{\infty})$ and $(u_x^n)_{n\in \mathbb{N}}$ is a Cauchy sequence in $L^{\infty}([0,T];L^2)$. For this purpose, set $u^{n+m+1}-u^{n+1}$, we obtain
\begin{equation}
    \partial_t(u^{n+m+1}-u^{n+1})+u^{n+m+1}\partial_x(u^{n+m+1}-u^{n+1})=-(u^{n+m}-u^{n})u^{n+1}_x+\int_{\infty}^{x}\frac{1}{2} (u^{n+m}_x-u^{n}_x)(u^{n+m}_x+u^{n}_x)dz.
\end{equation}

By virtue to Lemma \ref{priori estimate} and \eqref{yzyj1}, we have
\begin{align}\label{ineq111}
    \quad &\|u^{n+m+1}-u^{n+1}\|_{L^{\infty}} \leq C_{u_0}\int_0^t\|u^{n+m}-u^{n}\|_{\dot{H}^{1}\cap L^{\infty}}ds.
\end{align}
and
\begin{align}\label{ineq121}
    \|u^{n+m+1}-u^{n+1}\|_{\dot{H}^{1}}
    &=\|(u^{n+m+1}-u^{n+1})_x\|_{L^2}\notag\\
    &\leq C_{u_0}\int_0^t\|u^{n+m}-u^{n}\|_{\dot{H}^{1}\cap L^{\infty}}ds.
\end{align}
Combining \eqref{ineq111} and \eqref{ineq121} yields that
\begin{align}
    \|u^{n+m+1}-u^{n+1}\|_{\dot{H}^{1}\cap L^{\infty}}&\leq C_{u_0}\int_0^t\|u^{n+m}-u^{n}\|_{\dot{H}^{1}\cap L^{\infty}}ds \notag \\
    &\leq \frac{(TC_{u_0})^{n+1}}{(n+1)!} \|u^m\|_{E^2_{2,2}} \notag \\
    &\leq C_{u_0,T}2^{-n}
\end{align}
This implies $\{u^n\}_{n\in \mathbb{N}}$ is a Cauchy sequence in $C([0,T];L^{\infty}\cap \dot{H}^1)$. Since $(u^n)_{n\in \mathbb{N}}$ is bounded in $L^{\infty}([0,T];E^2_{2,2})$. Using interpolation inequalities implies that
$$u^n\rightarrow u\quad in\quad B^{\frac{3}{2}-\varepsilon}_{\infty,2},\quad \varepsilon >0$$
$$u^{n}_x\rightarrow \bar{u}_x\quad in\quad H^{1-\varepsilon},\quad \varepsilon >0$$
\quad\\

\textbf{Step four. Convergence.} \\
\par
We then prove that $u$ in $E^2_{2,2}$ and satisfies \eqref{eq1}. Applying the Fatou property we deduce that
\begin{equation}\label{fatou1}
  \left\{\begin{array}{l}
u^n\quad is\quad bounded\quad in\quad B^{\frac{3}{2}}_{\infty,2}\quad \Longrightarrow u_n \rightharpoonup u\in B^{\frac{3}{2}}_{\infty,2},\quad u_{nx} \rightharpoonup u_x \in B^{\frac{1}{2}}_{\infty,2}  \\
u^{n}_x\quad is\quad bounded\quad in\quad H^{1}\quad \Longrightarrow u_{nx} \rightharpoonup \bar{u}_x\in H^{1}\subset B^{\frac{1}{2}}_{\infty,2}.
  \end{array}\right.
\end{equation}
Because $H^{1}\subset B^{\frac{1}{2}}_{\infty,2}$, we get $\bar{u}_x=u_x$. In fact,
$$<\bar{u_x}-u_x,\varphi>=<\bar{u_x}-u^n_{x},u_x>+<u_x-u^n_{x},\varphi>\rightarrow 0,\quad \forall\varphi\in C^{\infty}_0.$$
Then we may pass to the limit in \eqref{eq2} easily and conclude that $u$ is indeed a solution of \eqref{eq1} in the sense of distributions.

Finally, as $u$ belongs to $L^{\infty}([0,T];E^2_{2,2})$, the right-hand side of \eqref{eq1} also belongs to $L^{\infty}([0,T];L^{\infty}\cap\dot{H}^1)$, which implies $u_t$ is in $C([0,T];L^{\infty}\cap\dot{H}^1)$. Similar to the proof of Lemma \ref{existence} in \cite{book}, we can easily deduce that $u$ belongs to $C([0,T];E^2_{2,2})\cap C^1([0,T];\bar{E}^2_{2,2})$.\\
\quad\\

\textbf{Step five. Uniqueness and continuous dependence.} \\
\par
We will prove the uniqueness of solutions to \eqref{eq1} next. This proof is based on the way we have in Step 3. Suppose that $(m_1, m_2)$ are two solutions of \eqref{eq1}, set $w=m_1-m_2$, we obtain
\begin{equation*}
    \partial_tw+m_1\partial_xw=-wm_{2x}+\int_{-\infty}^{x}\frac{1}{2} w_x(m_1+m_2)dz,
\end{equation*}

By virtue to Lemma \ref{priori estimate}, we have
\begin{align}\label{ineq11}
    \|w(t)\|_{L^{\infty}}\leq \|w(0)\|_{L^{\infty}}+ C_{u_0}\int_0^t\|w\|_{\dot{H}^{1}\cap L^{\infty}}ds.
\end{align}
and
\begin{align}\label{ineq12}
    \|w(t)\|_{\dot{H}^{1}}\leq \|w(0)\|_{\dot{H}^{1}}+ C_{u_0}\int_0^t\|w\|_{\dot{H}^{1}\cap L^{\infty}}ds.
\end{align}

Combining \eqref{ineq11}, \eqref{ineq12} and Gronwall's inequality yield that
\begin{align}\label{ineq18}
    \|w(t)\|_{\dot{H}^{1}\cap L^{\infty}}\leq \|w(0)\|_{\dot{H}^{1}\cap L^{\infty}}+C_{u_0}\int_0^t\|w\|_{\dot{H}^{1}\cap L^{\infty}}ds\leq C_{u_0}\|w(0)\|_{\dot{H}^{1}\cap L^{\infty}}.
\end{align}

Therefore, the uniqueness is obvious in view of \eqref{ineq18}.
Moreover, an interpolation argument ensures that the continuity with respect to the initial data holds for the norm $C([0,T];B^{\frac{3}{2}-\epsilon}_{\infty ,2}\cap \dot{H}^{2-\epsilon})$ for $\epsilon >0$ sufficient small. In fact, similar to the proof of \cite{li2}, we can raise the continuity with respect to the initial data until $C([0,T];E^2_{2,2})$.
\quad\\

\textbf{Step six. other cases.} \\

Since we have prove the local well-posedness for (\ref{eq1}) in $C([0,T];E^2_{2,2})\cap C^1([0,T];\bar{E}^2_{2,2})$, other cases are similar. In fact, we should only take some modifications in step three (step six is similar).

Consider the Cauchy sequence $w^{n+1}:=u^{n+m+1}-u^{n+1}$:
\begin{equation}\label{ineqnew18}
\frac{d}{dt}w^{n+1}+u^{n+1}w^{n+1}_x=-w^{n}u^{n+m+1}_x+\int_{\infty}^{x}w^{n+1}_x(u^{n+m}_x+u^{n}_x)dz
\end{equation}
Similar to Step 2, it's easy to deduce that $\{u^n_x\}$ is bounded in $L^q$ if $u_{x0}\in L^q$. Then by virtue to Lemma \ref{priori estimate} and the Holder inequality, we have
\begin{align}\label{ineqnew11}
\|w^{n+1}(t)\|_{L^{\infty}}
    &\leq
    C\int_{0}^{t}\|u^{n+m+1}_x\|_{L^{\infty}}\|w^{n+1}\|_{L^{\infty}}+
    \|u^{n+m}_x+u^{n}_x\|_{L^{q}} \|w^{n}\|_{\dot{W}^{1,p}}ds   \notag\\
    &\leq C_{u_0} \int_{0}^{t} \|w^{n}\|_{L^{\infty}\cap\dot{W}^{1,p}}ds
\end{align}

and
\begin{align}\label{ineqnew12}
    \|w^{n+1}(t)\|_{\dot{W}^{1,p}}
    &\leq C\int_{0}^{t}(\|u^{n+m+1}_x\|_{L^{\infty}}+\|u^{n}_x\|_{L^{\infty}})(\|w^{n}\|_{\dot{W}^{1,p}}+\|w^{n+1}\|_{\dot{W}^{1,p}})
    +(\|u^{n+m+1}_{xx}\|_{L^{p}})\|w^{n}\|_{L^{\infty}}ds  \notag\\
    &\leq C_{u_0}\int_{0}^{t} \|w^{n}\|_{L^{\infty}\cap\dot{W}^{1,p}}ds,
\end{align}
where the last inequality from the Gronwall's inequality. Then Combining (\ref{ineqnew11}) and (\ref{ineqnew12}), we get
\begin{align}
    \|w^{n+1}(t)\|_{L^{\infty}\cap\dot{W}^{1,p}} \leq C_{u_0}\int_{0}^{t} \|w^{n}\|_{L^{\infty}\cap\dot{W}^{1,p}}ds
\end{align}
After some calculations we still get $\{u^n\}_{n=1}^{\infty}$ is a Cauchy sequence in $L^{\infty}\cap\dot{W}^{1,p}$.

This complete the proof.
\end{proof}
\begin{rema}
For $1\leq p\leq 2$, we have $\dot{B}^{1}_{p,p}\cap\dot{B}^{2}_{p,p}\hookrightarrow\dot{W}^{1,q}$. But for $p> 2$, we have to consider it for an extra space $\dot{W}^{1,q}$ for the initial data.
\end{rema}

\section{Blow-up and global existence}
\par
First we prove a conservation inequality for \eqref{eq1}.
\begin{lemm}\label{conservation}
Let $u_0\in E^s_{2,2},s>\frac{5}{2}$ and $T^*$ be the maximal existence time of the corresponding solution $u$ to \eqref{eq1}, then we have
$$\|u(t)\|_{\dot{H}_1}\leq \|u_0\|_{\dot{H}_1}.$$
\end{lemm}

\begin{proof}
Differentialing the equation \eqref{eq1} and multiplying with $2u_x$, we have
\begin{equation}\label{eq3}
    \frac{d}{dt}u^2_x+(uu^2_x)_x=0.
\end{equation}
This implies
\begin{align*}
 \frac{d}{dt}\|u(t)\|^2_{\dot{H}_1}=0.
\end{align*}
\end{proof}

Next we state a blow-up criterion for \eqref{eq1}.
\begin{lemm}\label{blow up}
Let $u_0\in E^s_{2,2}=L^{\infty}\cap \dot{H}^1\cap\dot{H}^s$ with $s>\frac{5}{2}$ being as in Theorem \ref{theorem},
and let $T^*$ be the maximal existence time of the corresponding solution $u$ to \eqref{eq1}. Then $u$ blows up in finite time $T^*$ if and only if
$$ \int_0^{T^*} \|u_x(t')\|_{L^{\infty}}dt'=\infty.$$
\end{lemm}
\begin{proof}
Taking the $L^{\infty}$ norm to (\ref{eq1}) both sides, we have
\begin{equation}\label{ineq23}
    \|u(t)\|_{L^{\infty}}\leq C(\|u_0\|_{L^{\infty}}+\int_0^t \|u_x\|_{L^{\infty}}\|u\|_{L^{\infty}}
    +\|u_{0x}\|^2_{L^2}ds,
\end{equation}
Then differentiating the equation \eqref{eq1},
\begin{equation}\label{eq24}
    \frac{d}{dt}u_x+uu_{xx}=-\frac{1}{2}u^2_x.
\end{equation}
By virtue to Lemma \ref{priori estimate} and $u_x\in H^{s-1},s>\frac{5}{2}$, we get
\begin{equation}\label{ineq25}
    \|u_x(t)\|_{H^{s-1}}\leq C(\|u_0\|_{H^{s-1}}+\int_0^t \|u_x\|_{L^{\infty}}\|u_x\|_{H^{s-1}}ds,
\end{equation}
Combining (\ref{ineq23}), (\ref{ineq25}) and the Gronwall inequality, we have
\begin{equation}\label{ineq250}
    \|u(t)\|_{E^s_{2,2}}\leq Ce^{\int_{0}^{t}\|u_x\|_{L^{\infty}}ds}[(\|u_0\|_{E^s_{2,2}}+ \|u_{0x}\|^2_{L^2}].
\end{equation}
If $T^*$ is finite, and $\int_0^{T^*} \|u_x\|_{L^{\infty}}dt'<\infty$, then $u\in L^{\infty}([0,T^*];E^s_{2,2})$, which
contradicts the assumption that $T^*$ is the maximal existence time.

On the other hand, by Theorem \ref{theorem} and the fact that $E^s_{2,2}\hookrightarrow \dot{W}^{1,\infty}$, if $\int_0^{T^*} \|u_x\|_{L^{\infty}}dt'=\infty$, then $u$ must blow up in finite time.
\end{proof}
\begin{rema}
By the algebra interpolation, we can easily get a weaker blow-up criterion for \eqref{eq1}:
$$\lim_{t\rightarrow T}\left\|u_x\right\|_{B^{0}_{\infty ,\infty}}=\infty.$$
\end{rema}

Let us consider the ordinary differential equation:
\begin{equation}\label{eq4}
  \left\{\begin{array}{l}
  q_t(t,x)=u(t,q(t,x)),\quad t\in[0,T),  \\
  q(0,x)=x,\quad x\in\mathbb{R}.
  \end{array}\right.
\end{equation}
If $u\in E^s_{2,2}$ with $s\geq 2$ being as in Theorem \ref{theorem}, then $u\in C([0,T); C^{0,1})$. By the classical results in the theory of ordinary differential equations, we can easily infer that \eqref{eq4} have a unique solution $q\in C^1([0,T)\times\mathbb{R};\mathbb{R})$ such that the map $q(t,\cdot)$ is an increasing diffeomorphism of $\mathbb{R}$ with
$$q_x(t,x)=\exp\Big(\int_0^t u(t',q(t',x)\Big)dt'>0,\quad \forall (t,x)\in[0,T)\times\mathbb{R}.$$

Now the following theorem shows that under particular condition for the initial data, the corresponding solution of \eqref{eq1} will exist globally in time.
\begin{theo}\label{global}
Let $u_0\in E^s_{2,2}$, $s>\frac{5}{2}$. Assume $u_{0xx}(x)\geq 0$ when $x\leq x_0$, $u_{0xx}(x)\leq 0$ when $x\geq x_0$ for some $x_0\in\mathbb{R}$, and $u_{0x}(x_0)>0$.
Then the corresponding solution $u$ of \eqref{eq1} exists globally in time .
\end{theo}
\begin{proof}
Arguing by density, now we assume $s>\frac 7 2$.
Differentiating \eqref{eq24}, we have
\begin{equation}\label{eq5}
    \frac{d}{dt}u_{xx}+uu_{xxx}=-2u_xu_{xx}
\end{equation}
Then we have
$$\frac{d}{dt}u_{xx}(t,q(t,x))=-2u_xu_{xx},$$
Hence,
$$u_{xx}(t,q(t,x))=u_{0xx}\exp\Big(\int_0^t -2u_x(t',q(t',x))dt'\Big),$$
which implies that $u_{xx}$ doesn't change sign, so we can deduce that
\begin{equation}\label{eq6}
   u_{xx}(t,x)\geq 0,\quad when\quad x\leq q(t,x_0)
   \quad and \quad
   u_{xx}(t,x)\leq 0,\quad when\quad x\geq q(t,x_0).
\end{equation}
Moreover, we have
\begin{equation}\label{eq7}
    \frac{d}{dt}u_{x}(t,q(t,x))=-\frac{1}{2}u^2_x\leq 0.
\end{equation}

Using the fact that the flow $q(t,x)$ is a differmorphism and $u_{0x}(x_0)>0$, and by \eqref{eq6} we deduce that $q(t,x_0)$ is the maximum value point. \eqref{eq7} tells us $u_x(t)$ will decrease monotonly at the point $q(t,x_0)$ along the flow. Moreover, since $u_x(t)$ belongs to $H^{s-1}$, it will decrease but will not be less than zero. Otherwise, it will contradict with the decay of infinity in $H^s$.

As a result, we can deduce that the maximum point is at the initial point $x_0$ such that
$$\|u_x\|_{L^{\infty}}\leq \|u_{0x}\|_{L^{\infty}},$$
which means the solution is global by Lemma \ref{blow up} .
\end{proof}
\begin{rema}
The initial datas satisfying the conditions of Theorem \ref{global} are exist such as $u_0(x)=\int_{-\infty}^{x} e^{-z^2}dz$. It's easy to deduce that $u_0(x)\in L^{\infty}\cap \dot{H}^1\cap \dot{H}^2$, $u_{0x}(0)>0$, and $u_{0xx}(x)\geq 0$ when $x\leq 0$ and $u_{0xx}(x)\leq 0$ when $x\geq 0$.
\end{rema}

We then shows that the corresponding solution of \eqref{eq1} will blow up by giving some particular condition for the initial data.
\begin{theo}
Let $u_0\in E^s_{2,2}$, $s\geq 2$. Assume that there exists a point $x_0$ such that $u_{0x}(x_0)<0$,
Then the corresponding solution $u$ of \eqref{eq1} blows up in finite time.
\end{theo}
\begin{proof}
Arguing by density, now we assume $s>\frac{5}{2}$.
From (\ref{eq24}), we have
\begin{equation}\label{eq24}
    \frac{d}{dt}u_x(t,q(t,x))=-\frac{1}{2}u^2_x.
\end{equation}
Solving the above equality, we finally get
\begin{equation}\label{ineq10}
    u_x(t,q(t,x))=\frac{2}{t+\frac{2}{u_{0x}}}
\end{equation}
As $u_{0x}(x_0)< 0$, we can easily deduce that the maximal time $T<-\frac{2}{u_{0x}(x_0)}$.

Therefore, from \eqref{ineq10} we know $u_x(t)\rightarrow-\infty$ as $t\rightarrow T$. By Lemma \eqref{blow up}, the solution $u$ will blow up in finite time.
\end{proof}

\section{ill-posedness}
\par
In this section, we are going to prove the norm inflation in $A\triangleq L^{\infty}\cap \dot{B}^{\frac{1}{p}}_{p,r}\cap \dot{B}^{1+\frac{1}{p}}_{p,r}$ with $1\leq p\leq \infty,r>1$
\begin{theo}
Let $1<r\leq \infty ,1\leq p\leq \infty$. For any $\epsilon>0$,there exists $u_0\in H^{\infty}$ such that\\
(1)$\left\|u_0\right\|_{A}\leq \epsilon.$\\
(2)There exists a unique solution $u\in E^{\infty}_{p,r}$ with maximal $T<\epsilon$.\\
(3)$\lim_{t\rightarrow T}\left\|u_x\right\|_{B^{\frac{1}{p}}_{p,r}}\geq \lim_{t\rightarrow T}\left\|u_x\right\|_{B^{0}{\infty ,\infty}}=\infty.$
\end{theo}
\begin{proof}
We first define
$$h(x)=\sum_{k\geq 1}^{\infty}\frac{h_k(x)}{2^{2k}k^{\frac{2}{1+r}}}$$
where $h_k(\xi)=i2^{-k}\xi\varphi(2^{-k}\xi)$, and $\varphi(\xi)$ is even non-negative, non-zero $C^{\infty}_0$ function in \cite{G-L-Y}. Similar to the proof of \cite{G-L-Y}, we can easily prove that
$$\left\|h\right\|_{A}\leq\left\|h\right\|_{B^{1+\frac{1}{p}}_{p,r}}\leq C,\quad\quad h'(0)=-\infty$$

For any $\epsilon>0$ small enough, let $u_0:=\frac{\epsilon S_Nh}{\left\|h\right\|_{B^{1+\frac{1}{p}}_{p,r}}}$, where $N$ is sufficient large such that $u'_0(0)<-2\epsilon^{-1}$. Then $\left\|u_0\right\|_{A}\leq \epsilon$ and $u_0\in E^{\infty}_{p,r}$ for fixed $N$.

Finally, by Theorem \ref{theorem} and Theorem \ref{blow up}, we get a unique local solution $u\in E^{\infty}_{p,r}$ and $u(t,x)$ blow up in finite time T, we deduce that
$$\lim_{t\rightarrow T}\left\|u_x\right\|_{B^{\frac{1}{p}}_{p,r}}\geq \lim_{t\rightarrow T}\left\|u_x\right\|_{B^{0}_{\infty ,\infty}}=\infty.$$
Then we find the example:
$$\left\|u_0-0\right\|_{A}\leq \epsilon ,\quad\quad \left\|u(t)-0\right\|_{A}\rightarrow\infty$$
This implies the ill-posedness of (\ref{eq1}).
\end{proof}
\begin{rema}
For $s=2,p=r=1$, we get the local well-posedness of (\ref{eq1}) in $E^2_{1,1}=L^{\infty}\cap \dot{B}^{1}_{1,1}\cap \dot{B}^{1}_{1,1}$, the above theorem implies the solution will be ill-posedness in $E^2_{1,r},r>1$.

For $s=2,p=r=2$, we also get the local well-posedness of (\ref{eq1}) in $E^2_{2,2}=L^{\infty}\cap \dot{H}^1\cap \dot{H}^2$. However, the above theorem only implies the solution will be ill-posedness in $L^{\infty}\cap \dot{H}^1\cap \dot{H}^{\frac{3}{2}}$, we don't whether the solution will be well-posedness or ill-posedness in $L^{\infty}\cap \dot{H}^1\cap \dot{H}^{s},\frac{3}{2}\leq s<2$.
\end{rema}

\section{Unique continuation}
\par
In this final section, we consider the unique continuation of $(\ref{eq1})$ with $g(t)=C\int_{-\infty}^{\infty}u_x^2(z)dz,\quad C\in\mathbb{R}$. We first recall the equation :
\begin{equation}
  \left\{\begin{array}{l}
    u_{x}+uu_x=\int_{-\infty}^{x}\frac{1}{2}u_x^2(z)dz+C\int_{-\infty}^{\infty}u_x^2(z)dz.  \\
    u(0,x)=u_0(x), \quad x\in\mathbb{R},
  \end{array}\right.
\end{equation}

\begin{theo}
Let $u(t,x)$ be a real strong solution of (\ref{eq1}). If there exists a open set $\Omega=(a,b)\times[t_1,t_2]$, $(a,b\in\mathbb{R},t_1,t_2\geq 0)$ such that
$$u(t,x)=0,\quad\quad (x,t)\in\Omega ,$$
and $u(t,x)$ meets one of the following cases:\\
1) For $C>0$\\
2) For $C=0$, $b=\infty$\\
3) For $-\frac{1}{2}<C<0$, $b=\infty $ or $a=-\infty$\\
4) For $C=-\frac{1}{2}$, $a=-\infty$\\
5) For $C<-\frac{1}{2}$\\
then $u\equiv 0$,\quad $(x,t)\in\mathbb{R}^+\times\mathbb{R} $.
\end{theo}
\begin{proof}
1) For $C>0$, (\ref{eq1}) can be wrote as
$$u_{x}+uu_x=\int_{-\infty}^{x}\frac{1}{2}u_x^2(z)dz+C\int_{-\infty}^{\infty}u_x^2(z)dz.$$
From the hypothesis it follows that
$$u_{x}+uu_x=0,\quad\quad(x,t)\in\Omega .$$
This implies that
$$\int_{-\infty}^{x}\frac{1}{2}u_x^2(z)dz+C\int_{-\infty}^{\infty}\frac{1}{2}u_x^2(z)dz=0,\quad(x,t)\in\Omega .$$
As $\Omega=(a,b)\times[t_1,t_2]$, we can easily get that $u_x=0,\quad (x,t)\in\mathbb{R}\times[t_1,t_2]$. The continuity of the strong solution implies $u=0,\quad (x,t)\in\mathbb{R}\times[t_1,t_2]$. Then we get $u\equiv 0,\quad (x,t)\in\mathbb{R}\times\mathbb{R}^+$ by the uniqueness.\\

Other cases are similar to 1), we omit it here.
\end{proof}

\noindent\textbf{Acknowledgements.}
This work was partially supported by NNSFC (No. 11671407), FDCT (No. 098/2013/A3), Guangdong Special Support Program (No. 8-2015), and the key project of NSF of Guangdong province (No. 2016A03031104).

\addcontentsline{toc}{section}{\refname}

\end{document}